\documentclass[11pt,reqno]{amsart}

\usepackage{amssymb,latexsym,amsmath,amsfonts,amsthm}
\usepackage{mathrsfs}
\usepackage{enumitem}
\usepackage[usenames]{color}
\usepackage{hyperref}
\usepackage{comment}
\usepackage{tikz-cd}
\usepackage{graphicx}

\allowdisplaybreaks

\numberwithin{equation}{section}

\definecolor{DPurple}{rgb}{0.46,0.2,0.69}

\theoremstyle{definition}
\newtheorem{definition}{Definition}[section]

\newtheorem{problem}[definition]{Problem}

\theoremstyle{remark}
\newtheorem{remark}[definition]{Remark}

\theoremstyle{plain}
\newtheorem{theorem}[definition]{Theorem}
\newtheorem{result}[definition]{Result}

\newtheorem{proposition}[definition]{Proposition}

\newcommand{\eps}{\varepsilon}

\newcommand\leb[1]{\mathbb{L}^{{#1}}}

\newcommand{\bcdot}{\boldsymbol{\cdot}}
\newcommand{\lrarw}{\longrightarrow}
\newcommand{\btl}{\blacktriangleleft}
\newcommand{\bdy}{\partial}

\newcommand{\gaux}{\gamma^{{\rm aux}}}
\newcommand{\sigg}{\sigma_{\gamma}}
\newcommand{\hrm}{\boldsymbol{{\sf h}}}
\newcommand{\intv}{\mathscr{I}}

\newcommand{\Z}{\mathbb{Z}}
\newcommand{\N}{\mathbb{N}}

\newcommand{\Cn}{\mathbb{C}^n}

\newcommand{\R}{\mathbb{R}}

\newcommand{\re}{{\sf Re}}
\newcommand{\im}{{\sf Im}}

\newcommand{\wt}{\widetilde}

\begin{document}

\title[Non-smooth unit speed paths]{Non-smooth paths having unit speed with respect to \\
the Kobayashi metric, and their applications}

\author{Gautam Bharali}
\address{Department of Mathematics, Indian Institute of Science, Bangalore 560012, India}
\email{bharali@iisc.ac.in}

\author{Rumpa Masanta}
\address{Theoretical Statistics and Mathematics Unit, Indian Statistical Institute, Bangalore Centre,
Bangalore 560059, India}
\email{rumpa\_ra@isibang.ac.in}

\dedicatory{Dedicated to Prof.~Josip Globevnik on the occasion of his 80th birthday}

\begin{abstract}
We investigate the question of whether a non-constant absolutely continuous path can be
reparametrized to be of unit speed with respect to the Kobayashi metric and be absolutely continuous. Even
when the answer is ``Yes,'' which isn't always the case, its proof involves some subtleties. We answer the
above question and discuss several applications.
\end{abstract}

\keywords{Almost-geodesics, chord-arc curves, geodesics, Kobayashi length, notions of visibility}
\subjclass[2020]{Primary: 32F45, 32Q45; Secondary: 53C23}

\maketitle

\vspace{-7.5mm}
\section{Introduction}\label{S:intro}
Consider a complex manifold $X$ and assume that $X$ is Kobayashi hyperbolic (i.e., that the Kobayashi
pseudodistance $K_X$ is a distance). Let $\gamma:[0,T]\lrarw X$ be an absolutely continuous path. We introduce a function $g$ such that $(g(t)-g(s))$ gives the $k_X$-length of $\gamma|_{[s,t]}$, $s<t\in [0,T]$:
\[
  g(t):=\int_{0}^{t} k_X(\gamma(s);\gamma'(s))\,ds \quad \forall t\in[0,T],
\]
where $k_X$ denotes the Kobayashi metric on $X$. We defer to Section~\ref{S:met} the discussion on what it means
for $\gamma$ to be absolutely continuous in our general setting, what $\gamma'$ means, etc.
\smallskip

This paper is motivated by the following questions (in what
follows, if $\gamma:[0,T]\lrarw X$ is a path, then we will denote the image of $\gamma$ by
$\langle\gamma\rangle$):

\begin{problem}\label{pro:questions}
Let $X$ be a Kobayashi hyperbolic complex manifold. Let $\gamma: [0,T]\lrarw X$ be a non-constant absolutely continuous path.
\begin{itemize}[leftmargin=25pt]
  \item[$(a)$] Can $\gamma$ be reparametrized with respect to its $k_{X}$-length as 
  an absolutely continuous path?
  \item[$(b)$] Does there exist an absolutely continuous path $\Gamma: [0,\tau]\lrarw  X$ such that
  $\langle\gamma\rangle = \langle\Gamma\rangle$ and can be reparametrized with respect to
  its $k_{X}$-length as an absolutely continuous path?
\end{itemize}
\end{problem}

A part of the motivation for Problem~\ref{pro:questions} is that this
seemingly simple problem runs into measure-theoretic subtleties, some of which aren't immediately
evident; see Remark~\ref{rem:subtlety}. Unsurprisingly, the answer to $(a)$ is, in general,
``No.'' But even when the answer to $(a)$ is ``Yes,'' establishing this is non-trivial
because we insist on the reparametrization to be absolutely continuous. Here, one of
the obstacles is evident: even if $g$ is
invertible, $g^{-1}$ is not always absolutely continuous. In this regard,
Problem~\ref{pro:questions} is \emph{very different} from its Riemannian analogue.
\smallskip

A deeper motivation for
Problem~\ref{pro:questions} is that its solution would be a very useful
tool\,---\,analogous to unit-speed reparametrization in Riemannian geometry\,---\,leading to a range
of applications. The applications that interest us relate to notions of negative curvature for the
metric space $(X,K_X)$. The first notion is that of Gromov hyperbolicity of $(X,K_X)$, on which there has
been a lot of work recently: see, for instance,
\cite{baloghbonk:GhKmspd00, zimmer:GhKmcdft16, fiacchi:GhpfdC222}. The second notion is that of
visibility, which is a notion that involves domains $\Omega\varsubsetneq X$, with $\Omega$ being Kobayashi
hyperbolic. Visibility is a weak notion of negative curvature introduced by Bharali--Zimmer for Kobayashi
hyperbolic domains $\Omega\varsubsetneq\Cn$  \cite{bharalizimmer:gdwnv17, bharalizimmer:gdwnv23}\,---\,also
see \cite{bharalimaitra:awnovfoewt21} by Bharali--Maitra. Very recently, this notion was extended to Kobayashi
hyperbolic domains $\Omega\varsubsetneq X$ \cite{masanta:vdekdcm24}, $X$ being any complex manifold. 
Very roughly, the visibility property requires that all geodesics with end-points close to two distinct points
in $\bdy\Omega$ must bend uniformly into $\Omega$. For a \emph{non-compact} Kobayashi hyperbolic complex
manifold $X$, where $\dim_{\mathbb{C}}(X)\geq 2$, it is, in general, unclear when the metric space $(X,K_X)$
is Cauchy-complete\,---\,even when $X\varsubsetneq\Cn$ is a pseudoconvex domain\,---\,and, consequently,
whether $(X, K_{X})$ is a geodesic space. Thus, to study either notion of negative curvature
when $(X,K_X)$ is \textbf{not} assumed to be Cauchy-complete, one works with a class of
absolutely continuous quasi-geodesics; it is known that, given $x,y\in X$, $x\neq y$, there always 
exists such a quasi-geodesic joining $x$, $y$ \cite{bharalizimmer:gdwnv17, masanta:vdekdcm24}. This is where
absolutely continuous paths enter the discourse. We will elaborate upon this presently and discuss
a few of the above-mentioned applications in detail in Section~\ref{S:appl}. But, first, let us state a
theorem that answers the questions in Problem~\ref{pro:questions}.

\begin{theorem}[\textsc{Reparametrization Lemma}]\label{T:unit-speed}
Let $X$ be a Kobayashi hyperbolic complex manifold. Let $\gamma: [0,T]\lrarw X$ be a non-constant absolutely
continuous path. Then, there exists an absolutely continuous path $\sigg: [0,\tau] \lrarw X$ such that
$\langle\gamma\rangle = \langle\sigg\rangle$ and such that
$k_{X}(\sigg(t); \sigg'(t)) = 1$
for almost every $t\in [0,\tau]$. Writing $\mathscr{S}(\gamma) := \{t\in [0,T] : \gamma'(t) = 0\}$:
\begin{itemize}[leftmargin=25pt]
  \item[$(a)$] $\gamma$ can be reparametrized with respect to its $k_{X}$-length as an absolutely
  continuous path if and only if $\mathscr{S}(\gamma)$ contains no intervals of positive length.
  \item[$(b)$] When $\mathscr{S}(\gamma)$ contains no intervals of positive length, the path
  $\sigg$ is obtained by reparametrizing $\gamma$ with respect to its $k_{X}$-length.
\end{itemize}
\end{theorem}

The connection between Theorem~\ref{T:unit-speed} and the notions of negative curvature mentioned above are
given by \emph{$(\lambda,\kappa)$-chord-arc curves} and \emph{$(\lambda,\kappa)$-almost-geodesics}\,---\,see
Section~\ref{S:met} for definitions. The role of $(\lambda,\kappa)$-chord-arc curves is well established in
the study of Gromov hyperbolic spaces. For $X$ as described at the top of this section,
given a $\kappa>0$, for each pair $x,y\in X$, $x\neq y$, there exists a
$(1,\kappa)$-almost-geodesic joining $x$ and $y$; see \cite[Proposition~4.4]{bharalizimmer:gdwnv17},
\cite[Proposition~2.8]{masanta:vdekdcm24}. These serve as
substitutes for geodesics when $(X, K_X)$ is not a geodesic space. It would be useful to know whether
the two classes of paths coincide in some
appropriate sense. Answering this is one of the applications of Theorem~\ref{T:unit-speed};
see Section~\ref{S:appl}.
\smallskip

Another context in which $(\lambda,\kappa)$-chord-arc curves feature is in the study of
$(\lambda,\kappa)$-visible points of $\bdy\Omega$ for a Kobayashi hyperbolic domain $\Omega\varsubsetneq \Cn$.
This context provides a concrete instance of the usefulness of knowing whether the class of
$(\lambda,\kappa)$-chord-arc curves coincides with the class of $(\lambda,\kappa)$-almost-geodesics in any
sense. A point $p\in \bdy\Omega$ is a $(\lambda,\kappa)$-visible point, $\lambda\geq 1$ and $\kappa\geq 0$, if,
roughly speaking, all $(\lambda,\kappa)$-chord-arc curves originating sufficiently close to $p$ initially bend
away from $\bdy\Omega$ uniformly. This is a notion introduced by Nikolov \emph{et al.} 
\cite{nikolovoktenthomas:lgnvwrKdc24}. For the goals of \cite{nikolovoktenthomas:lgnvwrKdc24}, it is
essential to show that all $(\lambda,\kappa)$-chord-arc curves (which are called 
\emph{$(\lambda,\kappa)$-geodesics} in \cite{nikolovoktenthomas:lgnvwrKdc24}) are
$(\lambda,\kappa)$-almost-geodesics in some appropriate sense. Nikolov \emph{et al.} sketch how to do so
in \cite[Section~2]{nikolovoktenthomas:lgnvwrKdc24}
via a argument that claims that when $\Omega$ is Kobayashi
hyperbolic, then any absolutely continuous path in $\Omega$ can be reparametrized with respect to its
$k_{\Omega}$-length as an absolutely continuous path. In this, they appeal to Bharali--Zimmer
\cite{bharalizimmer:gdwnv17}. However:
\begin{itemize}[leftmargin=30pt]
  \item Given the observations above in connection with part~$(a)$ of 
  Problem~\ref{pro:questions}, the
  above-mentioned argument in \cite{nikolovoktenthomas:lgnvwrKdc24} is in need of further details.
  \item The arguments of Bharali--Zimmer in \cite{bharalizimmer:gdwnv17} do not suffice to provide these
  details.
\end{itemize}
The details are provided by our Reparametrization Lemma: i.e., Theorem~\ref{T:unit-speed}.
See Section~\ref{S:appl} for the precise relationship between the class of
$(\lambda,\kappa)$-chord-arc curves and the class of $(\lambda,\kappa)$-almost-geodesics. We should
mention that this relationship was first established in a preliminary, and unpublished, version of this
work; see \cite[Corollary~1.3]{bharalimasanta:nsphuswrtKm24}.
\smallskip

The proof of Theorem~\ref{T:unit-speed} is given in Section~\ref{S:unit-speed}. We also have some other
applications of Theorem~\ref{T:unit-speed}, separate from the preceding discussion. These are presented
in Section~\ref{S:appl}.
\medskip

\section{Metrical preliminaries}\label{S:met}
We shall first elaborate upon several concepts mentioned in Section~\ref{S:intro}
whose definitions had been deferred. Our notation is borrowed liberally from \cite{masanta:vdekdcm24}.

\begin{definition}\label{D:abs-cont}
Let $X$ be a complex manifold of dimension $n$. A path $\gamma: I\lrarw X$, where $I\subseteq \R$ is an
interval, is said to be \emph{locally absolutely continuous} if for each $t_0\in I$ and each
holomorphic chart $(U,\varphi)$ around $\gamma(t_0)$, with $I(\varphi,t_0)$ denoting any closed
and bounded interval in
$\gamma^{-1}(U)$ containing $t_0$, $\varphi\circ\gamma|_{I(\varphi,t_0)}$ is absolutely continuous as a path
in $\R^{2n}$. The path $\gamma: I\lrarw X$ is said to be \emph{absolutely continuous} if
$I\varsubsetneq \R$ is a closed and bounded interval and $\gamma$ is locally absolutely continuous.
\end{definition}

From the fact that every locally absolutely continuous path in $\R^{2n}$ is almost-everywhere
differentiable, it follows from Definition~\ref{D:abs-cont} that the same is true for the paths defined
therein. For such a path $\gamma: I\lrarw
X$, if $t\in I$ is such that $D\gamma(t)$ exists, then there is a canonical identification of the vector
$D\gamma(t)1\in T_{\gamma(t)}X$ with a vector in $T{^{(1,0)}_{\gamma(t)}}X$. We denote the latter by
$\gamma'(t)$.
\smallskip

We can now define the two classes of paths mentioned several times in Section~\ref{S:intro}

\begin{definition}\label{D:almost-geodesic}
Let $X$ be a Kobayashi hyperbolic complex manifold. Let $I\subseteq \R$ be an interval. For
$\lambda\geq 1$ and $\kappa\geq 0$, a path $\sigma: I\lrarw X$ is called
a \emph{$(\lambda, \kappa)$-almost-geodesic} if 
\begin{itemize}[leftmargin=25pt]
 \item[$(a)$] for all $s,t\in I$
 \[
   \frac{1}{\lambda}|s-t|-\kappa\leq K_X(\sigma(s),\sigma(t))\leq \lambda|s-t|+\kappa,
 \]
 \item[$(b)$] $\sigma$ is locally absolutely continuous (whence $\sigma'(t)$ exists for almost
 every $t\in I$),
 and for almost every $t\in I$, $k_X(\sigma(t);\sigma'(t))\leq\lambda$.
\end{itemize}
\end{definition}

\begin{definition}\label{D:chord-arc}
Let $X$ be a Kobayashi hyperbolic complex manifold. Let $I\subseteq \R$ be an interval. For
$\lambda\geq 1$ and $\kappa\geq 0$, a path $\sigma: I\lrarw X$ is called
a \emph{$(\lambda, \kappa)$-chord-arc curve} if $\sigma$ is locally absolutely continuous and
\[
  l_{X}(\sigma|_{[s,t]}) \leq \lambda K_{X}(\sigma(s), \sigma(t)) + \kappa
  \quad \text{for all $[s,t]\subseteq I$.}
\]
Here, $l_{X}(\bcdot)$ denotes the $k_{X}$-length of the arc in question.
\end{definition}

We must define one other term mentioned in Section~\ref{S:intro}.

\begin{definition}\label{D:geod}
Let $X$ be a Kobayashi hyperbolic complex manifold. Let $I\subseteq \R$ be an interval. A path
$\sigma: I\lrarw X$ is called a \emph{geodesic} if
\[
  K_X(\sigma(s),\sigma(t)) = |s-t|
\]
for all $s,t\in I$.    
\end{definition}

The following class of paths will feature in our discussion of applications in Section~\ref{S:appl}  

\begin{definition}\label{D:hyp-geo}
Let $X$ be a Kobayashi hyperbolic complex manifold. Let $I\subseteq \R$ be an interval. A path
$\sigma: I\lrarw X$ is called a \emph{length-minimizing geodesic} or a \emph{hyperbolic geodesic} if $\sigma$
is locally absolutely continuous and
\[
  K_X(\sigma(s), \sigma(t)) = \int_s^t k_X(\sigma(u); \sigma'(u))\,du \quad
  \text{for all $[s,t]\subseteq I$.}
\]
\end{definition}

As hinted at in Section~\ref{S:intro}, it is not always the case that for any pair of points $x,y\in X$,
with $x\neq y$\,---\,where $X$ is as in the last three definitions\,---\,there exists a geodesic joining
$x$ and $y$. An analogous statement can be made about the existence of length-minimizing geodesics joining
$x$ and $y$. The following results address these issues. In what follows (and in later sections), the
word ``path'' will refer to a continuous map of an interval $I\subseteq \R$ into the topological space
featured in the discussion.
\smallskip

Given a topological space $X$, we refer the reader to
\cite[Sections~2.1 and~2.2]{buragoburagoivanov:cmg01} for a definition of a class of
\emph{admissible paths} in $X$, say $\mathscr{A}$, and of a \emph{length structure} $(X,\mathscr{A}, L)$ on
$X$, where $L: \mathscr{A}\lrarw [0, +\infty]$ is understood to be the length of each path 
$\gamma\in \mathscr{A}$. If $(X,\mathscr{A}, L)$ is such that
\begin{enumerate}[leftmargin=25pt]
  \item for each $x,y\in X$ with $x\neq y$, there exists a path $\gamma: [0,T]\lrarw X$ in $\mathscr{A}$
  such that $\gamma(0)=x$ and $\gamma(T)=y$; and
  \item $L: \mathscr{A}\lrarw [0,+\infty)$;
\end{enumerate}
then the function $d_L: X\times X\lrarw [0,+\infty)$ given by
\[
  d_L(x,y) := \inf\{L(\gamma)\mid \gamma: [0,T_{\gamma}]\lrarw X \text{ belongs to $\mathscr{A}$ and }
  \gamma(0)=x, \ \gamma(T_{\gamma})=y\}
\]
is (owing to the definition of admissibility) a distance on $X$. A metric space $(X,d)$ is called a
\emph{length space} if $d$ equals the distance $d_L$ associated with some length
structure $(X,\mathscr{A}, L)$ with the properties~$(1)$ and~$(2)$ above.
\smallskip

We have the following fundamental theorem:

\begin{result}[paraphrasing {\cite[Theorem~2.5.23]{buragoburagoivanov:cmg01}}]\label{R:l_sp}
Let $(X,d)$ be a length space and let $(X,\mathscr{A}, L)$ be an associated length structure.
Assume $(X,d)$ is Cauchy-complete and locally compact. Then, given any $x,y\in X$, $x\neq y$, there
exists a path $\sigma: [0,T]\lrarw X$ in $\mathscr{A}$ with $\sigma(0)=x$ and $\sigma(T)=y$ such that
$L(\sigma|_{[s,t]}) = d(\sigma(s), \sigma(t))$ for all $s, t\in [0,T]$ with $s\leq t$.
\end{result}

\begin{remark}
Given $(X,d)$ and $x,y\in X$ as in the above statement, \cite[Theorem~2.5.23]{buragoburagoivanov:cmg01}
asserts the existence of a $\sigma: [0,T]\lrarw X$ in $\mathscr{A}$ with $\sigma(0)=x$ and $\sigma(T)=y$
such that $L(\sigma) = d(x,y)$. The seemingly stronger conclusion of Result~\ref{R:l_sp} is an automatic
consequence of the properties of $\mathscr{A}$; see
\cite[pages~26 and~48]{buragoburagoivanov:cmg01}.
\end{remark}

It follows from results in \cite{venturini:pprcm89} that if $X$ is a Kobayashi
hyberbolic manifold, then $(X,\mathscr{A}_{AC},l_{X})$\,---\,where $\mathscr{A}_{AC}$ is the class of all absolutely
continuous paths in $X$\,---\,is a length structure and the Kobayashi distance $K_X=d_{l_X}$.
Moreover, the topology induced by $K_{X}$ is the usual topology on $X$, due to which 
$(X, K_X)$ is locally compact. Thus, from Result~\ref{R:l_sp}, we conclude:

\begin{proposition}\label{P:l_minim}
Let $X$ be a Kobayashi hyperbolic manifold and suppose $(X,K_X)$ is Cauchy-complete. Then, given any
$x,y\in X$, $x\neq y$, there exists a length-minimizing geodesic joining $x$ and $y$.
\end{proposition}

The above proposition can also be proved directly \cite{bracci:pers} without appealing to
Result~\ref{R:l_sp}. Interestingly, this relies upon the Reparametrization Lemma (i.e.,
Theorem~\ref{T:unit-speed}).
\smallskip

If we start with a metric space $(X,d)$ with the property that
any two points $x,y\in X$, $x\neq y$, can be joined by a path in $X$ that is rectifiable with respect to
$d$, then if $\mathscr{A}_r$ is
the class of all rectifiable paths in $X$ and, for a rectifiable path $\gamma: [0,T]\lrarw X$, we define
\[
  L_{d}(\gamma) := \sup\Big\{\sum\nolimits_{j=1}^N d(\gamma(s_j), \gamma(s_{j-1})):
                    0=s_0 < s_1 < s_2 <\dots < s_N=T
                    \Big\},
\]
then $(X,\mathscr{A}_r, L_d)$ is a length structure with the properties~$(1)$ and~$(2)$ above. If
$(X,d)$ is a length space where $d$ arises from the length structure $(X,\mathscr{A}_r, L_d)$, then
$(X,d)$ is called a \emph{length metric space}. With these definitions, we can state the following
version of the Hopf--Rinow theorem (see \cite[Chapter~I.3]{bridsonhaefliger:msnpc99}, for instance):
\smallskip

\begin{result}[Hopf--Rinow]\label{R:H-R}
Let $(X,d)$ be a length metric space. If $(X, d)$ is locally compact, then the following are equivalent:
\begin{itemize}[leftmargin=25pt]
  \item[$(a)$] Every bounded set in $X$ is relatively compact. 
  \item[$(b)$] $(X,d)$ is Cauchy-complete. 
\end{itemize}
Either of the above properties implies that given any $x,y\in X$, $x\neq y$, there exists a geodesic of
$(X, d)$ joining $x$ and $y$.
\end{result}

A geodesic of $(X,d)$ is a path having the property stated in Definition~\ref{D:geod} with $d$ replacing
$K_X$. Now, for exactly the same reasons stated prior to Proposition~\ref{P:l_minim}, Result~\ref{R:H-R}
implies the following:

\begin{proposition}\label{P:K_H-R}
Let $X$ be a Kobayashi hyperbolic manifold. The following are equivalent:
\begin{itemize}[leftmargin=25pt]
  \item[$(a)$] Every $K_X$-bounded set in $X$ is relatively compact. 
  \item[$(b)$] $(X,K_X)$ is Cauchy-complete. 
\end{itemize}
Either of the above properties implies that given any $x,y\in X$, $x\neq y$, there exists a geodesic
joining $x$ and $y$.
\end{proposition}

Lastly, we present a result that is used in exactly one place in our proofs but plays an important and interesting role.

\begin{result}[paraphrasing {\cite[Theorem~2]{royden:rkm71}}]\label{R:kob-hyp-cond}
Let $X$ be a complex manifold. Fix a Hermitian metric $\hrm$ on $X$ and let $d_{\hrm}$ denote the
distance induced by $\hrm$. Then, $X$ is Kobayashi hyperbolic if and only if for each
$x\in X$, there exists a constant $c_{x}>0$ and a neighbourhood $U_{x}\ni x$ such that
$k_X(y;v)\geq c_{x}\hrm_y(v)$ for every $v\in T{^{(1,0)}_{y}}X$ and every $y\in U_x$.  
\end{result}
\smallskip

\section{Measure-theoretic preliminaries}~\label{S:measure-theo}
In this section, we gather several standard results in the theory of the Lebesgue measure on $\R$ (with just the
first result needing, perhaps, a remark on its proof). Here, and in subsequent sections, we shall abbreviate the
words ``for almost every $x$'' to ``for a.e. $x$''.

\begin{result}\label{R:deriv}
Let $a<b$ and let $f: [a,b]\lrarw \R$ be a monotone increasing function. Then $f'(x)$ exists for a.e.
$x\in (a,b)$, $f'$ is Borel measurable, and $f'$ is of class $\leb{1}([a,b])$.
\end{result}

\begin{remark}
The usual statement of Result~\ref{R:deriv} focuses on the fact that $f'$ is Lebesgue measurable.
However, its \emph{proof} (see \cite[pages~111--113]{wheedenzygmund:mi77}, for
instance) shows that the set
$B:=\{x\in (a,b) : f'(x) \text{ does not exist}\}$ is a Borel set. Moreover, on
$(a,b)\setminus B$, $f'$ is equal to any one of the Dini derivates (restricted to
$(a,b)\setminus B$), which are Borel measurable. Thus, $f'$, on extending to $[a,b]$ in the usual
way, is Borel measurable. 
\end{remark}

\begin{result}\label{R:zero_meas}
Let $a<b$ and let $f$ be as in Result~\ref{R:deriv}. If $f$ is absolutely continuous, then
$f$ maps sets of (Lebesgue) measure zero to sets of measure zero.
\end{result}

See the proof of \cite[Theorem~7.18]{rudin:RCa87} for a proof of the above result. The final result of this section
is a change-of-variables formula. Since ``change of variables'' can mean different things in different contexts, we
state the formula relevant to this work.

\begin{result}\label{R:change_of_var}
Let $a<b$ and let $\varphi: [a,b]\lrarw [c,d]$ be a monotone increasing, absolutely continuous function such that
$\varphi(a)=c$ and $\varphi(b)=d$. Let $f$ be a non-negative Lebesgue measurable function defined on $[c,d]$. Then
\[
  \int_{c}^d f(u)\,du = \int_{a}^b f\big(\varphi(x)\big)\varphi'(x)\,dx.
\]
\end{result}

The above version of the change-of-variables formula is given in \cite[Chapter~7]{rudin:RCa87}.
\smallskip

\section{The key propositions}\label{S:key}
This section is devoted to a pair of results that constitute the proof of Theorem~\ref{T:unit-speed}. For any set
$S\subseteq \R$ that is Lebesgue measurable, $|S|$ will denote its Lebesgue measure.
\smallskip

The obstacle to reparametrizing the path appearing in Theorem~\ref{T:unit-speed} by its $k_{X}$-length becomes evident in the following result.

\begin{proposition}\label{P:aux}
Let $X$ be a Kobayashi hyperbolic complex manifold of dimension $n$. Let $\gamma: [0,T]\lrarw X$ be a non-constant
absolutely continuous path. Let
\[
  \mathscr{S}(\gamma) := \{t\in [0,T] : \gamma'(t) = 0\},
\]
and suppose $|\mathscr{S}(\gamma)| > 0$. If $\mathscr{S}(\gamma)$ contains intervals of positive length, then there
exists an auxiliary path $\gaux: [0,\tau]\lrarw X$ (where $\tau<T$), satisfying
$\langle\gaux \rangle = \langle\gamma \rangle$, that is absolutely continuous and such that the set
$\{t\in [0,\tau] : (\gaux)'(t) = 0\}$ has empty interior. 
\end{proposition}
\begin{proof}
Let us assume that $\mathscr{S}(\gamma)$ has non-empty interior.
We define an equivalence relation on $[0,T]$ as follows:
\[
  s\thicksim t \iff \left.\gamma\right|_{[\min\{s,t\},\,\max\{s,t\}]} \ \text{is constant}.
\]
Let $[0,T]/\!\thicksim$ denote the quotient space and $\pi: [0,T]\lrarw [0,T]/\!\thicksim$ be the quotient map.
By the continuity of $\gamma$, if, for a point $p\in [0,T]/\!\thicksim$, $\pi^{-1}\{p\}$ has more than one element,
then $\pi^{-1}\{p\}$ is a closed interval of positive length. Let
\[
  \mathscr{C}(\gamma) := \{\pi^{-1}\{p\}\varsubsetneq [0,T]: p\in [0,T]/\!\thicksim \text{ and } 
  							\pi^{-1}\{p\} \text{ has more than one element}\}.
\]
By our assumption on $\mathscr{S}(\gamma)$ and as $\gamma$ is absolutely continuous, 
$\mathscr{C}(\gamma)\neq \emptyset$.
\smallskip 
 
It is well-known\,---\,see \cite[Section~2.5]{buragoburagoivanov:cmg01}, for instance\,---\,that
there exists a constant $\tau>0$ (recall that $\gamma$ is non-constant), a homeomorphism
$h: ([0,T]/\!\thicksim) \lrarw [0,\tau]$ and a continuous path $\gaux: [0, \tau]\lrarw X$ such that the
diagram
\begin{center}
  \begin{tikzcd}[column sep=large, row sep=large]
  \left[0,T\right]	        \arrow[r, "\gamma"]
  			           \arrow[d, "h\circ \pi"']
  &X \\
  \left[0,\tau\right]	\arrow[ur, "\gaux"']
  \end{tikzcd}
\end{center}
commutes and such that:
\begin{enumerate}[leftmargin=25pt]
  \item For any interval $L\subseteq [0,\tau]$, $(h\circ \pi)^{-1}(L)$ is an interval.
  \item For any interval $L\subseteq [0,\tau]$ and any interval $I\in \mathscr{C}(\gamma)$
  such that $(h\circ \pi)^{-1}(L)\setminus I\neq \emptyset$, if
  $C(I,L)$ denotes the set of connected components of $(h\circ \pi)^{-1}(L)\setminus I$, then
  \begin{equation}\label{E:cut_length}
    {\rm length}(L) = \sum\nolimits_{S\in C(I,L)}{\rm length}(S).
  \end{equation}
\end{enumerate}
(By the property $(1)$, $(h\circ \pi)^{-1}(L)$ is an interval, whereby $C(I,L)$ has
at most two elements. Thus, the right-hand side of \eqref{E:cut_length} is a finite sum.) It
follows that $\langle\gaux \rangle = \langle\gamma \rangle$.  
\smallskip
       
We shall now prove that $\gaux$ is absolutely continuous. The idea behind the proof is, essentially, that
since $\gaux$ is constructed by gluing together restrictions of $\gamma$ on various subintervals of $[0,T]$,
$\gaux$ must be absolutely continuous. To elaborate on this, fix $t_0\in [0,T]$ and a holomorphic chart
$(U,\varphi)$ around $\gamma(t_0)$. In what follows, $J$ will denote the set $\{1, 2,\dots, N\}$ for
some $N\in \Z_+$ whose value would depend on the discussion at hand.
As $\gamma$ is absolutely continuous, if we fix
$\intv:= I(\varphi,t_0)$\,---\,a closed and bounded interval contained in $\gamma^{-1}(U)$\,---\,then
given $\eps>0$, there exists a number $\delta(\eps)>0$ such that for any finite collection
$\{(a_j,b_j): j\in J\}$ of pairwise disjoint open subintervals of $\intv$ such that 
$\sum_{j\in J}(b_j-a_j) < \delta(\eps)$, we have
\begin{align}
  \sum_{j\in J}|\re(\varphi_k\circ\gamma)(b_j) - \re(\varphi_k\circ\gamma)(a_j)| 
  &< \eps, \label{E:real-part_k}\\
  \sum_{j\in J}|\im(\varphi_k\circ\gamma)(b_j) - \im(\varphi_k\circ\gamma)(a_j)|
  &< \eps, \label{E:im-part_k}
\end{align}
for each $k=1,\dots, n$.
\smallskip

For any $I\in \mathscr{C}(\gamma)$, $I\cap \mathbb{Q}\neq \emptyset$. From this and the fact that if
$I_1\neq I_2\in \mathscr{C}(\gamma)$, then $I_1\cap I_2 = \emptyset$, it follows that $\mathscr{C}(\gamma)$
is at most countable. We shall now assume that $\mathscr{C}(\gamma)$ has countably
many intervals (the proof is, essentially, routine when $\mathscr{C}(\gamma)$ is finite).
Note that since $\intv$ was arbitrarily chosen subject to the conditions mentioned in the previous paragraph,
it suffices to show that
\begin{itemize}
  \item[$(*)$] $\varphi\circ\gaux|_{\intv^*}$ is an absolutely continuous path in $\R^{2n}$,
  where $\intv^* := h\circ \pi(\intv)$.
\end{itemize}
Fix $\eps>0$. With $\delta(\eps)$ as given above, let $\{(\alpha_j,\beta_j): j\in J\}$ be a
finite collection of pairwise disjoint open subintervals of $\intv^*$ such that 
$\sum_{j\in J}(\beta_j-\alpha_j) < \delta(\eps)/2$. Fix $j^*\in J$ and write $L:= (\alpha_{j^*}, \beta_{j^*})$.
Let $C(L)$ denote the set of elements (i.e., intervals) of $\mathscr{C}(\gamma)$ that intersect
$(A_{j^*}, B_{j^*}):= (h\circ \pi)^{-1}(L)$ (that $(h\circ \pi)^{-1}(L)$ is an interval is due to $(1)$
above). At this stage, the argument leading to the inequalities (\ref{E:telescope_re}, \ref{E:telescope_im})
below splits into several cases. We
shall consider the case where:
\begin{itemize}[leftmargin=30pt]
  \item $A_{j^*}$ is in none of the intervals in $C(L)$ while $B_{j^*}$ is contained in some
  interval in $C(L)$;
  \item $C(L)$ is countably infinite.
\end{itemize}
Other cases can be analysed similarly to deduce (the appropriate form of) the inequalities
(\ref{E:telescope_re}, \ref{E:telescope_im}) below. Now, let $\{I_{\nu}: \nu\in \Z_+\}$
be an enumeration of $C(L)$. Write
\[
 m_{\nu} := \inf(I_{\nu}\cap (A_{j^*}, B_{j^*})), \quad M_{\nu}:= \sup(I_{\nu}\cap (A_{j^*}, B_{j^*})).
\]
Next, define the numbers
\begin{align*}
  \wt{M}_{\nu} &:= \inf\{m_{\mu}: \mu\in \Z_+ \text{ and } m_{\mu} > M_{\nu}\}, \; \; \nu\in \Z_+, \\
  \wt{M}_0 &:= \inf\{m_{\nu}: \nu\in \Z_+\}.
\end{align*}
Then, owing to the definitions of $m_{\nu}, M_{\nu}$, $\nu=1, 2, 3,\dots$,
\[
  \{(A_{j^*}, \wt{M}_0)\}\cup \{(M_{\nu}, \wt{M}_{\nu}): \nu\in \Z_+\}
  \]
is a countable collection of pairwise disjoint open intervals, some of which \emph{could
be empty sets.}
As these are all contained in $(A_{j^*}, B_{j^*}) = (h\circ\pi)^{-1}(L)$, it
follows from \eqref{E:cut_length} by an elementary inductive argument
that
\[
  (\wt{M}_0-A_{j^*}) + \sum_{\nu=1}^\infty(\wt{M}_{\nu} - M_{\nu}) = {\rm length}(L).
\]
Due to this, and in view of the fact that (writing $M_0:=A_{j^*}$) 
$\inf\{(\wt{M}_{\nu} - M_{\nu}): \nu\in \N\} = 0$, we can find a number $N(j^*)\in \Z_+$ and points
\[
  A_{j^*}= c_1 < \wt{c}_1 < c_2 < \wt{c}_2 < \dots < c_{N(j^*)} < \wt{c}_{N(j^*)} < B_{j^*}\,,
\]
where $\wt{c}_{i}\in (\{\wt{M}_{\nu}: \nu\in \N\}\cap \{m_{\nu}: \nu\in \Z_+\})$ for each
$i=1,\dots, N(j^*)$, $\wt{c}_{N(j^*)}$ is such that
$[\,\wt{c}_{N(j^*)}, B_{j^*}]\subseteq I_{\nu}$ for some $\nu\in \Z_+$, 
and if $\wt{c_i} = m_{\nu_i}$, then $c_{i+1} = M_{\nu_i}$ for each
$i=1,\dots, N(j^*)-1$, such that
\[
  \sum_{i=1}^{N(j^*)}(\wt{c}_i-c_i) < \big(1+2^{-(j^*+1)}\big){\rm length}(L).
\]
From the above and the fact that
$\gamma|_{[\,\wt{c}_{i},\,c_{i+1}]}$ is constant, we easily conclude:
\begin{align}
  |&\re(\varphi_k\circ\gaux)(\beta_{j^*}) - \re(\varphi_k\circ\gaux)(\alpha_{j^*})| \notag \\
  &=|\re(\varphi_k\circ\gamma)(B_{j^*}) - \re(\varphi_k\circ\gamma)(A_{j^*})|\,
  \leq \sum_{i=1}^{N(j^*)}|\re(\varphi_k\circ\gamma)(\wt{c}_{i}) - \re(\varphi_k\circ\gamma)(c_{i})|, 
  \label{E:telescope_re}\\
  |&\im(\varphi_k\circ\gaux)(\beta_{j^*}) - \im(\varphi_k\circ\gaux)(\alpha_{j^*})| \notag \\
  &=|\im(\varphi_k\circ\gamma)(B_{j^*}) - \im(\varphi_k\circ\gamma)(A_{j^*})|\,
  \leq\sum_{i=1}^{N(j^*)}|\im(\varphi_k\circ\gamma)(\wt{c}_{i}) - \im(\varphi_k\circ\gamma)(c_{i})|,
  \label{E:telescope_im}
\end{align}    
for each $k=1,\dots, n$. In the above argument, $j^*\in J$
was arbitrarily chosen. Thus, for each $L\in \{(\alpha_j,\beta_j): j\in J\}$ (equivalently,
$j\in J$) we obtain a finite collection, say $\mathscr{C}(j)$,
of pairwise disjoint open subintervals of $\intv$
\[
  \text{the sum of whose lengths} < \big(1+2^{-(j+1)}\big){\rm length}(L) < 2\bcdot{\rm length}(L)
\]
and associated to which we have estimates analogous to the inequalities
(\ref{E:telescope_re}, \ref{E:telescope_im}). As $\bigcup_{j\in J}\mathscr{C}(j)$ is a finite collection
of pairwise disjoint open subintervals of $\intv$ the sum of whose lengths is less than $\sum_{j\in J}2(\beta_j-\alpha_j)
< \delta(\eps)$,
it follows from the last two assertions and (\ref{E:real-part_k}, \ref{E:im-part_k}) that
\begin{align*}
  \sum_{j\in J}|\re(\varphi_k\circ\gaux)(\beta_j) - \re(\varphi_k\circ\gaux)(\alpha_j)| 
  &< \eps, \\
  \sum_{j\in J}|\im(\varphi_k\circ\gaux)(\beta_j) - \im(\varphi_k\circ\gaux)(\alpha_j)|
  &< \eps, \label{E:im-part_k}
\end{align*}
for each $k=1,\dots, n$. This establishes absolute continuity of $\gaux$ (recall Definition~\ref{D:abs-cont}).
\smallskip

Finally, it follows from the commutative diagram above that $\gaux$ cannot be constant on
any subinterval of $[0,\tau]$ of positive length. Thus, as we have shown that $\gaux$ is absolutely
continuous, the set $\{t\in [0,\tau] : (\gaux)'(t) = 0\}$ must have empty interior.
\end{proof}

We now present our second key result.

\begin{proposition}\label{P:unit_speed}
Let $X$ be a Kobayashi hyperbolic complex manifold.
Let $\Gamma: [0,\tau]\lrarw X$ be an absolutely continuous path 
and define
\[
  \mathscr{S}(\Gamma) := \{t\in [0,\tau] : \Gamma'(t) = 0\}.
\]
Assume that $\mathscr{S}(\Gamma)$ has empty interior. Write
\begin{equation}\label{E:Gamma_length}
  G(t) := \int_{0}^t k_{X}(\Gamma(s); \Gamma'(s))\,ds \quad \forall t\in [0, \tau].
\end{equation}
Then:
\begin{itemize}[leftmargin=25pt]
  \item[$(a)$] $G$ is strictly increasing.
  \item[$(b)$] Writing $h := G^{-1}$, $\sigma := \Gamma\circ h$, and
  $\ell := G(\tau)$, $\ell < \infty$, $\sigma'(t)$ exists for a.e.
  $t\in [0, \ell]$, $[0, \ell]\ni t\mapsto (\sigma(t), \sigma'(t))$ determines
  a $T^{(1,0)}X$-valued Borel-measurable map, and we have:
  \begin{equation}\label{E:key_eqn}
    t = \int_{0}^t k_{X}(\sigma(u); \sigma'(u))\,du \quad \forall t\in [0, \ell].
  \end{equation}
  \item[$(c)$] The path $\sigma$ satisfies $k_{X}(\sigma(t); \sigma'(t)) = 1$ for
  a.e. $t\in [0, \ell]$ and $\sigma$ is absolutely continuous
\end{itemize}
\end{proposition}

\begin{remark}
Even though it is not \emph{a priori} known whether $\sigma$ is absolutely continuous, the right-hand side of
\eqref{E:key_eqn} makes sense in view of the statement preceding it and because $k_{X}$ is upper semi-continuous
(owing to which $k_{X}(\sigma(\bcdot); \sigma'(\bcdot))$ is a non-negative measurable function).
\end{remark}

\begin{proof}
Since $\Gamma$ is absolutely continuous, by \cite[Corollary~4.1]{venturini:pprcm89}
$k_{X}(\Gamma(\bcdot); \Gamma'(\bcdot))$ is of class $\leb{1}([0,\tau])$. Thus, $\ell < \infty$. 
Given a function $F: [0,\tau]\lrarw \R$, one of the characterizations of the absolute continuity
of $F$ is the existence of a function $f$ of class $\leb{1}([0,\tau])$ such that
$F(t) = F(0) + \int_{0}^{t}f(s)\,ds$ for all $t\in [0,\tau]$. Therefore, as
$k_{X}(\Gamma(\bcdot); \Gamma'(\bcdot))$ is of class $\leb{1}([0,\tau])$, $G$ is absolutely continuous.
\smallskip

By Lebesgue's differentiation theorem, we get $G'(t) = k_{X}(\Gamma(t); \Gamma'(t))$ for a.e. $t\in [0,\tau]$. So, as 
$k_{X}$ is upper semi-continuous and $X$ is Kobayashi hyperbolic, if $G$ were constant on an interval
$I\subseteq [0,\tau]$ of positive length, then $\Gamma'|_{I}\equiv 0$. Thus, $(a)$ now follows since $G$ is monotone
increasing and, by hypothesis, there are no intervals positive length in $\mathscr{S}(\Gamma)$.
\smallskip

Note that
\[
\sigma'(t) = h'(t)\Gamma'(h(t)) \quad \forall t\in [0, \ell] \text{ at which the right-hand side exists.}
\]
Let us write:
\[
  A_1 := \{t\in [0, \tau] : \Gamma'(t) \text{ does not exist}\} \ \ \text{and}
  \ \ B_1 := \{t\in [0, \ell] : h'(t) \text{ does not exist}\}.
\]
Then,
\begin{equation}\label{E:bad_set}
 \{t\in [0, \ell] : \sigma'(t) \text{ does not exist}\} = h^{-1}(A_1)\cup B_1 = G(A_1)\cup B_1.
\end{equation}
As $G$ is absolutely continuous, by Result~\ref{R:zero_meas}, $|G(A_1)| = 0$. By Result~\ref{R:deriv}, $|B_1| = 0$.
Thus, in view of \eqref{E:bad_set}, $\sigma'(t)$ exists for a.e. $t\in [0, \ell]$. Thus,
\[
  (\sigma(t), \sigma'(t)) = \big(\Gamma\circ h(t), h'(t)\Gamma'(h(t))\big) \quad \text{for a.e.} \ t\in [0, \ell].
\]
Since $\Gamma$ is absolutely continuous, the measurability claimed for $t\mapsto (\sigma(t), \sigma'(t))$
follows from the above identity and 
Result~\ref{R:deriv}. Note that if $h$ were absolutely continuous, then this appeal to
Result~\ref{R:deriv} would not have been needed. However, $h$ need not be absolutely continuous
(see Remark~\ref{rem:subtlety}). As $k_{X}$ is upper semi-continuous, 
$k_{X}(\sigma(\bcdot); \sigma'(\bcdot))$ is a non-negative Borel measurable\,---\,and, hence, Lebesgue
measurable\,---\,function. Hence, the integrals on the left-hand side of \eqref{E:key-eqn} below exist (even though
they may \emph{a priori} equal $+\infty$). As $G$ is absolutely continuous,
by Result~\ref{R:change_of_var}, for every $t\in [0,\ell]$ we have
\begin{align}
  \int_{0}^t k_{X}(\sigma(u); \sigma'(u))\,du
  &= \int_{0}^{G^{-1}(t)}k_{X}\big(\sigma(G(s)); \sigma'(G(s))\big)G'(s)\,ds \label{E:key-eqn}\\
  &= \int_{0}^{G^{-1}(t)}k_{X}\big(\sigma(G(s)); \sigma'(G(s))G'(s)\big)\,ds\notag \\
  &= \int_{0}^{G^{-1}(t)}k_{X}\big(\Gamma(s); \Gamma'(s)\big)\,ds\label{E:chain_rule}\\
  &= G(G^{-1}(t))=t. \notag
\end{align}
This establishes \eqref{E:key_eqn}. Note that the equality in \eqref{E:chain_rule} follows from the relation
$\Gamma = \sigma\circ G$ and from the chain rule applied to this composition.
\smallskip

Now, applying Lebesgue's
differentiation theorem to \eqref{E:key_eqn}, we have $k_{X}(\sigma(t); \sigma'(t)) = 1$ for a.e.
$t\in [0, \ell]$. It remains to show that $\sigma$ is absolutely continuous. This is immediate when $h$ is absolutely
continuous. As $h$ need not, in general, be absolutely continuous (as alluded to above), a proof is
needed. To this end, let us fix a Hermitian metric $\hrm$ on $X$. It follows from 
Result~\ref{R:kob-hyp-cond}, since $\sigma([0, \ell])$ is compact, that there exists a constant
$c > 0$ such that
\[
  k_X(\sigma(u); \sigma'(u))\geq c\hrm_{\sigma(u)}(\sigma'(u)) \quad \text{for a.e.} \ u\in [0, \ell].
\]
Fix $t_0\in [0, \ell]$ and a holomorphic chart $(U, \varphi)$ around $\sigma(t_0)$. Fix a compact
interval $I(\varphi, t_0) =: I$ such that $\varphi\circ\sigma(I)\subset U$. As $\sigma(I)$ is
compact, there exists a constant $C > 1$ such that
\[
  C^{-1}\|(\varphi\circ\sigma)'(u)\| \leq \hrm_{\sigma(u)}(\sigma'(u)) \leq 
  C\|(\varphi\circ\sigma)'(u)\| 
  \quad \text{for a.e.} \ u\in I.
\]
Then, from the last two estimates:
\begin{align*}
  \|\varphi\circ\sigma(s) - \varphi\circ\sigma(t)\|
  &\leq \int_{s}^t\|(\varphi\circ\sigma)'(u)\|\,du \\
  &\leq C\!\int_{s}^t\hrm_{\sigma(u)}(\sigma'(u))\,du
  \leq \left(\frac{C}{c}\right)\!\int_{s}^tk_{X}(\sigma(u); \sigma'(u))\,du
  = \left(\frac{C}{c}\right)|s-t| 
\end{align*}
for every $s, t\in I$ such that $s<t$. In the above estimate, the last equality is due to the fact that
$k_{X}(\sigma(u); \sigma'(u)) = 1$ for a.e. $u\in I$. This shows that 
$\varphi\circ\sigma|_{I(\varphi, t_0)}$ is Lipschitz. As $t_0$ and $(U, \varphi)$ were arbitrarily chosen,
$\sigma$ is absolutely continuous. Hence, the proof of $(c)$ is complete. 
\end{proof}
\smallskip

\section{The proof of Theorem~\ref{T:unit-speed} and a related remark}\label{S:unit-speed}
We begin with

\begin{proof}[The proof of Theorem~\ref{T:unit-speed}]
Let $\tau := T$ if $\mathscr{S}(\gamma)$ contains no intervals of positive length, else let $\tau$ be as given by
Proposition~\ref{P:aux}. Next, define
\begin{equation}\label{E:Gamma}
  \Gamma(t) := \begin{cases}
                 \gamma(t),
                 &\text{if $\mathscr{S}(\gamma)$ contains no intervals of positive length,} \\
                 \gaux(t),
                 &\text{if $\mathscr{S}(\gamma)$ contains intervals of positive length,} 
            \end{cases}
\end{equation}
for $t\in [0,\tau]$, where $\gaux$ is as given by Proposition~\ref{P:aux}. We have
$\langle\gamma\rangle = \langle\Gamma\rangle$. Furthermore, by Proposition~\ref{P:aux},
$\mathscr{S}(\Gamma) := \{t\in [0,\tau] : \Gamma'(t) = 0\}$ contains no intervals of positive length.
Thus, if we write $\sigg := \Gamma\circ(G^{-1})$, where $G$ is as given by \eqref{E:Gamma_length}, then the desired properties of $\sigg$ follow from Proposition~\ref{P:unit_speed}.
\smallskip

That $\gamma$ can be reparametrized with respect to its $k_{X}$-length if $\mathscr{S}(\gamma)$
contains no intervals of positive length follows from Proposition~\ref{P:unit_speed}. The converse is
due to the fact that the function
\begin{equation}\label{E:g}
  g(t) := \int_{0}^{t} k_X(\gamma(s);\gamma'(s))\,ds \quad \forall t\in[0,T],
\end{equation}
is not invertible when $\mathscr{S}(\gamma)$ contains intervals of positive length. This establishes~$(a)$.
If $\mathscr{S}(\gamma)$ contains no intervals of positive length, then the description
of $\Gamma$ given by \eqref{E:Gamma} and the definition $\sigg := \Gamma\circ(G^{-1})$
establishes~$(b)$.
\end{proof}

\begin{remark}\label{rem:subtlety}
We can now highlight the subtleties, hinted at in Section~\ref{S:intro}, encountered
in answering the questions in Problem~\ref{pro:questions}.
These subtleties are of a fundamental nature. E.g., note: even when the function $g$ given
by \eqref{E:g} is invertible, $g^{-1}$ is not always absolutely continuous. Of course, we
need to deal with the version of this problem that arises for general $\gamma$. 
Let $\Gamma$ be as given by \eqref{E:Gamma}, $G$ be as given by \eqref{E:Gamma_length}, and
write $h = G^{-1}$. It is well-known that $h$ is absolutely continuous if and only if
$|\{t\in [0,\tau] : G'(t) = 0\}| = 0$. The latter condition fails if
$|\mathscr{S}(\Gamma)| >0$. When $h$ is not
absolutely continuous, we are faced with two difficulties:
\begin{itemize}[leftmargin=30pt]
  \item Firstly, we must determine whether $(\Gamma\circ h)'(t)$ exists for a.e. $t$ in the domain of $h$.
  \item Assuming that the above point is settled in the affirmative, the question still remains whether
  $k_{X}(\Gamma\circ h(\bcdot); (\Gamma\circ h)'(\bcdot))$ is Lebesgue measurable.
\end{itemize}
Analysing $k_{X}(\Gamma\circ h(\bcdot); (\Gamma\circ h)'(\bcdot))$ 
is vital to showing that $\sigg := \Gamma\circ h$ is a unit-speed path with respect to $k_{X}$. The
second question is fundamental to defining the integrals in \eqref{E:key-eqn}. It arises (when $h$
is not absolutely continuous) because a Borel measurable function pre-composed with a Lebesgue, but
not Borel, measurable map is not necessarily Borel measurable. This question
is resolved by Result~\ref{R:deriv} because $h$ is monotone increasing: due to this, and as
$k_{X}$ is upper semi-continuous, $k_{X}(\Gamma\circ h(\bcdot); (\Gamma\circ h)'(\bcdot))$ is Borel measurable.
\end{remark}
\smallskip

\section{Applications}\label{S:appl}
This section is dedicated to some applications of the Reparametrization Lemma.
Each of the theorems in this section\,---\,Theorems~\ref{T:C-A_A-G}, \ref{T:equiv},
and~\ref{T:loc_weak-geod}\,---\,will
be introduced by a brief discussion on the broader
questions in the Kobayashi geometry of \emph{non-compact} complex manifolds to which they relate.
The proofs of Theorems~\ref{T:C-A_A-G} and~\ref{T:equiv} will be presented in
Section~\ref{S:appl_proofs}.
\smallskip

\subsection{The relationship between almost-geodesics and chord-arc curves}
Our first application of the Reparametrization Lemma states, loosely speaking, that if $X$ is a Kobayashi
hyperbolic manifold, then, given any $(\lambda,\kappa)$-almost-geodesic $\gamma: [0,T]\lrarw X$, where
$\lambda\geq 1$ and $\kappa\geq 0$, there is a path $\sigma_{\gamma}$ having the same image as $\gamma$ 
and a pair $(\lambda', \kappa')$ that depends on $(\lambda, \kappa)$ such that $\sigma_{\gamma}$
is a $(\lambda', \kappa')$-chord-arc curve, and vice versa. A tangible instance of the need for such a result
was mentioned in Section~\ref{S:intro}. 

\begin{theorem}\label{T:C-A_A-G}
Let $X$ be a Kobayashi hyperbolic complex manifold.
\begin{itemize}[leftmargin=25pt]
  \item[$(a)$] Let $\gamma: [0,T]\lrarw X$ be a $(\lambda,\kappa)$-chord-arc curve for some $\lambda\geq 1$ and
  $\kappa\geq 0$. Then, there exists a path $\sigg$ in $X$ with $\langle\gamma\rangle = \langle\sigg\rangle$
  such that $\sigg$ is a $(\lambda, \kappa)$-almost-geodesic.
  \item[$(b)$] Let $\gamma: [0,T]\lrarw X$ be a $(\lambda,\kappa)$-almost-geodesic for
  some $\lambda\geq 1$ and
  $\kappa\geq 0$. Then, $\gamma$ is a $(\lambda^2,\lambda^2\kappa)$-chord-arc curve.
\end{itemize} 
\end{theorem}

\begin{remark}
Since the paths considered in Theorem~\ref{T:C-A_A-G} belong to special classes of paths that are
geometrically defined, the question arises whether one needs to appeal to the rather hard
Theorem~\ref{T:unit-speed}, which pertains to \emph{arbitrary} absolutely continuous paths, to
prove Theorem~\ref{T:C-A_A-G}. When, in part~$(a)$, $\kappa>0$,
it is not clear that $\mathscr{S}(\gamma)$ (in the notation of Theorem~\ref{T:unit-speed}) does
\textbf{not} contain intervals of positive length. In any event, there seems to be no easy way to
establish that the difficulties highlighted in Remark~\ref{rem:subtlety} cannot occur in part~$(a)$.
\end{remark}

Recalling the discussion in Section~\ref{S:intro} about the article
\cite{nikolovoktenthomas:lgnvwrKdc24}, we should mention an approximation strategy provided by
\cite[Proposition~2.7]{chandelgoraimaitrasarkar:vp1sva24} for associating a compact $(\lambda,\kappa)$-chord-arc
curve $\gamma$ with some almost-geodesic\,---\,not necessarily having the same image as $\gamma$\,---\,that would
serve the needs of \cite{nikolovoktenthomas:lgnvwrKdc24}. But it seems more natural,
given a $(\lambda,\kappa)$-chord-arc curve $\gamma: [0,T]\lrarw X$, to avoid approximations and be able to
produce a path $\sigma_{\gamma}$ with the same image as $\gamma$ such that $\sigma_{\gamma}$ is
a $(\lambda,\kappa)$-almost-geodesic. This appears to be the intuition of the authors of
\cite{nikolovoktenthomas:lgnvwrKdc24}. Part~$(a)$ of Theorem~\ref{T:C-A_A-G} validates
this intuition (and also extends it to a much more general setting). Theorem~\ref{T:C-A_A-G} 
appeared as Corollary~1.3 in a preliminary, unpublished, version \cite{bharalimasanta:nsphuswrtKm24} of this
work .
\smallskip

\subsection{Various notions of visibility}
We begin this section with a definition of the visibility property alluded to in Section~\ref{S:intro}.
Given a Kobayashi hyperbolic complex manifold $X$, if $(X, K_X)$ is Cauchy-complete, then we will say
that $X$ is \emph{Kobayashi complete}.

\begin{definition}\label{D:visible}
Let $X$ be a complex manifold and $\Omega\varsubsetneq X$ be a Kobayashi hyperbolic domain.
\begin{itemize}[leftmargin=25pt]
 \item[$(a)$] Let $p,q\in\bdy\Omega$, $p\neq q$. We say that the pair $(p,q)$ satisfies the
 \emph{weak visibility condition}
 if there exist neighbourhoods $U_p$ of $p$ and
 $U_q$ of $q$ in $X$ such that
 $\overline{U_p}\cap\overline{U_q}=\emptyset$ and such that for each $\kappa\geq 0$, there
 exists a compact set $K\subset \Omega$ such that the image of each $(1,\kappa)$-almost-geodesic
 $\sigma:[0,T]\lrarw\Omega$ with $\sigma(0)\in U_p$ and $\sigma(T)\in U_q$ intersects $K$.
 \item[$(b)$] We say that $\bdy\Omega$ is \emph{weakly visible} if every pair of points $p,q\in\bdy\Omega$,
 $p\neq q$, satisfies the weak visibility condition. 
\end{itemize}
\end{definition}

We refer the reader to Section~\ref{S:intro} as to why, given two distinct points $p$ and $q$ in a
Kobayashi hyperbolic complex manifold and a number $\kappa>0$, there exists a
$(1,\kappa)$-almost-geodesic joining $p$ and $q$. In a Kobayashi complete complex manifold, $p$ and $q$
can be joined by a geodesic as well as by a length-minimizing geodesic, thanks to
Propositions~\ref{P:K_H-R} and~\ref{P:l_minim}, respectively. Thus, with $X$, $\Omega$, and 
$p,q\in \bdy\Omega$ as in Definition~\ref{D:visible}, if $\Omega$ is Kobayashi complete, then:
\begin{itemize}[leftmargin=30pt]
  \item we say that the pair $(p,q)$ satisfies the \emph{geodesic visibility condition} and
  $\bdy\Omega$ is \emph{geodesically visible} to signify conditions analogous to those in
  Definition~\ref{D:visible} but featuring geodesics. \smallskip
  \item we say that the pair $(p,q)$ satisfies the \emph{hyperbolic visibility condition} and
  $\bdy\Omega$ is \emph{hyperbolically visible} to signify conditions analogous to those in
  Definition~\ref{D:visible} but featuring length-minimizing geodesics.
\end{itemize}

The notion of visibility in the sense of Definition~\ref{D:visible} has been of considerable interest lately.
Weak visibility of $\bdy\Omega$ enables one to control the behaviour of holomorphic maps from various classes
of domains into $\Omega$, leading to a host of results on the properties of such maps: see, for
instance, \cite{bharalizimmer:gdwnv17, bharalimaitra:awnovfoewt21, braccinikolovthomas:vkgcdrp22,
bharalizimmer:gdwnv23, nikolovoktenthomas:lgnvwrKdc24, masanta:vdekdcm24}. An example of such a result is
the extension of the classical Wolff--Denjoy theorem to progressively larger classes of domains far
\emph{beyond} bounded convex domains in $\Cn$, $n\geq 2$ (which, until recently, had presented the widest scope
for generalizing the Wolff--Denjoy theorem); see \cite{bharalizimmer:gdwnv17, bharalizimmer:gdwnv23,
vergamini:ttwpvc23, masanta:vdekdcm24}. Being similar to weak visibility, the notion of geodesic
visibility has its share of applications involving Kobayashi complete domains; see
\cite{braccinikolovthomas:vkgcdrp22, braccigaussiernikolovthomas:lgvkd23, chandelmaitrasarkar:nvwrkdca21}.
\smallskip

Length-minimizing geodesics are sometimes very useful to work with in the context of the applications alluded
to above because of the intuitions from Riemannian or Hermitian geometry that accompany them. The work
\cite{braccibenini:DWtscd24} by Bracci--Benini, for instance, makes delicate use of length-minimizing
geodesics. It also features
the use of the visibility property where, in the relevant claims, it refers to results on \emph{geodesic}
visibility. Now, it is fairly easy to see that on simply-connected planar hyperbolic domains, which are
the focus of \cite{braccibenini:DWtscd24}, geodesic visibility and the notion of visibility
involving length-minimizing geodesics (i.e., hyperbolic visibility) coincide. It would be useful
to extend, if possible, this agreement between two notions of visibility to any Kobayashi complete domain in
a general complex manifold. With this in mind, we show:

\begin{theorem}\label{T:equiv}
Let $X$ be a complex manifold and let $\Omega\varsubsetneq X$ be a Kobayashi complete domain.
Then:
 \[
   \bdy\Omega \text{ is weakly visible}\iff
   \bdy\Omega \text{ is geodesically visible }\iff 
   \bdy\Omega \text{ is hyperbolically visible.}
 \]
\end{theorem}

\subsection{Localization of weak visibility}
Theorem~\ref{T:equiv} is summarized as stating that when $X$ is a complex manifold and
$\Omega\varsubsetneq X$ is a Kobayashi complete domain, various notions of visibility of
$\bdy\Omega$ encountered in the recent literature coincide. But as weak visibility of $\bdy\Omega$ also
makes sense when $\Omega$ is not Kobayashi complete, we single out
weak visibility in exploring the natural question: \emph{under what conditions does $\bdy\Omega$ being
locally weakly visible imply that $\bdy\Omega$ is weakly visible?} This question calls for
a definition. In this definition, the words ``with respect to $U_p\cap\Omega$'' signify that
the property stated is as described by Definition~\ref{D:visible} with $U_p\cap\Omega$
replacing $\Omega$.

\begin{definition}
Let $X$ be a complex manifold and let $\Omega\varsubsetneq X$ be a Kobayashi hyperbolic domain.
\begin{itemize}[leftmargin=25pt]
  \item[$(a)$] Let $p\in\bdy\Omega$. We say that 
  $\bdy\Omega$ is \emph{locally weakly visible at $p$} if there exists a neighbourhood $U_p$ of $p$ in
  $X$ such that $U_p\cap\Omega$ is connected and is a Kobayashi hyperbolic domain, and there exists a
  neighbourhood $V_p$ of $p$ in $X$, $V_p\subseteq U_p$, such that every pair $(q_1, q_2)$ of distinct points in $V_p\cap\bdy\Omega$ satisfies the weak visibility condition with respect to
  $U_p\cap\Omega$.
 \item[$(b)$] We say that $\bdy\Omega$
 is \emph{locally weakly visible} if $\bdy\Omega$ is locally weakly visible at each point $p\in\bdy\Omega$.
\end{itemize}     
\end{definition}

We shall not discuss in detail the status of the question raised above for reasons that will become
clear. Instead, we present the following theorem. While it is of independent interest, its relevance to
the above question is evident.

\begin{theorem}\label{T:loc_weak-geod}
Let $X$ be a complex manifold and $\Omega\varsubsetneq X$ be a Kobayashi hyperbolic domain.
Fix an open subset $U$ of $X$ such that $U\cap \Omega$ is a non-empty connected set. Let
$V\Subset U$ be an open set with $V\cap \Omega\neq \emptyset$ such that
\begin{itemize}[leftmargin=30pt]
  \item $K_{\Omega}(V\cap \Omega, \Omega\setminus U) > 0$;
  \item every pair $(q_1, q_2)$ of distinct points in $V\cap\bdy\Omega$ satisfies the weak visibility condition with respect to $U\cap\Omega$.
\end{itemize}
For any open set $W\Subset V$ with $W\cap \Omega\neq \emptyset$ and any $\kappa\geq 0$, there
exists a $\kappa_0>0$ depending only on $U, V, W$, and $\kappa$ such that for any
$(1,\kappa)$-almost-geodesic $\gamma$ relative to $K_{\Omega}$ with 
$\langle\gamma \rangle\subset W\cap \Omega$, there exists a path $\sigma_{\gamma}$ satisfying
$\langle\gamma \rangle = \langle\sigma_{\gamma} \rangle$ that is a $(1,\kappa_0)$-almost-geodesic
relative to $K_{U\cap \Omega}$.
\end{theorem}

We shall merely indicate the proof of Theorem~\ref{T:loc_weak-geod}, highlighting its reliance on
Theorem~\ref{T:unit-speed}. This is because Theorem~\ref{T:loc_weak-geod} has implicitly been
established in \cite{masanta:vdekdcm24}, but with a key ingredient of its proof\,---\,the Reparametrization
Lemma\,---\,cited to an unpublished manuscript; i.e., to a preliminary version of this 
work \cite{bharalimasanta:nsphuswrtKm24}. Here is a brief outline of the proof of
Theorem~\ref{T:loc_weak-geod}:
\begin{itemize}[leftmargin=30pt]
  \item Let $\gamma: [0,T]\lrarw W\cap \Omega$ be a $(1,\kappa)$-almost-geodesic relative to $K_{\Omega}$.
  The hypotheses above allow us to appeal to \cite[Lemma~3.3]{masanta:vdekdcm24}, which establishes that:
  \[
    l_{U\cap \Omega}(\gamma|_{[s,t]}) \leq K_{U\cap \Omega}(\gamma(s), \gamma(t)) + \kappa_0
    \quad \text{for all $[s,t]\subseteq [0,T]$.}
  \]
  \item The previous step establishes that $\gamma$ is a $(1,\kappa_0)$-chord-arc curve. The theorem
  now follows from part~$(a)$ of Theorem~\ref{T:C-A_A-G}.
\end{itemize}

An application of Theorem~\ref{T:loc_weak-geod} is the answer to the question stated above in italics. While
not explicitly stated in \cite{masanta:vdekdcm24}, a suitable adaptation of Theorem~\ref{T:loc_weak-geod}
is used to provide a complete answer; see \cite[Theorem~1.6]{masanta:vdekdcm24}.
\medskip

\section{Propositions concerning visibility conditions}\label{S:prop_visi}
This section is dedicated to establishing a couple of propositions that are essential to proving
Theorem~\ref{T:equiv}. In order to state these propositions, we need to fix some notation and to state a
definition, which makes sense on any metric space.
\smallskip

Let $(X, d)$ be a metric space. For any
$x,y,o\in X$, we define the \emph{Gromov product} $(x|y)_o$ by $(x|y)_o :=(d(x,o)+d(y,o)-d(x,y))/2$. Now,
specializing to $\Omega\varsubsetneq X$, $X$ and $\Omega$ as in Theorem~\ref{T:equiv}: we shall denote the
Gromov product with the underlying distance being $K_{\Omega}$ by $(x|y)_o^{\Omega}$.
\smallskip

\emph{Some} of the cues to the results below were obtained from Section~2 of
\cite{braccinikolovthomas:vkgcdrp22} by Bracci~\emph{et al.} The most relevant results in
\cite[Section~2]{braccinikolovthomas:vkgcdrp22} are stated for bounded, Kobayashi complete domains in
$\Cn$. We extend these to Kobayashi complete domains in general complex manifolds. The connection
between our extended results and \cite{braccinikolovthomas:vkgcdrp22} is elaborated upon in the
proof of Proposition~\ref{P:BNT-type_1}.
\smallskip

Having introduced the notation required, we can present the following two propositions. 

\begin{proposition}\label{P:BNT-type_1}
Let $X$ be a complex manifold and let $\Omega\varsubsetneq X$ be a Kobayashi complete domain.
\begin{itemize}[leftmargin=25pt]
  \item[$(a)$] Let $p,q\in\bdy\Omega$, $p\neq q$. If the pair $(p,q)$ satisfies the geodesic
  visibility condition, then for any $o\in \Omega$,
  \[
    \limsup_{(x,y)\to (p,q)} (x|y)_o^{\Omega} < \infty.
  \]
  \item[$(b)$] Suppose that for some (hence any)
  $o\in \Omega$, $\limsup_{(x,y)\to (p,q)} (x|y)_o^{\Omega} < \infty$ for any $p,q\in\bdy\Omega$, $p\neq q$. Then,
  $\bdy\Omega$ is weakly visible. 
\end{itemize}
\end{proposition}
\begin{proof}
The poof of part~(a) is essentially the same as that of \cite[Proposition~2.4]{braccinikolovthomas:vkgcdrp22}.
While this result has been established for Kobayashi complete domains in $\Cn$, it extends verbatim to our
setting because the proof of \cite[Proposition~2.4]{braccinikolovthomas:vkgcdrp22} relies just on the definition
of a geodesic and on viewing $\Omega$ merely as a metric space: namely $(\Omega, K_{\Omega})$.
\smallskip

We therefore turn to part~$(b)$. It is, this time, not analogous to a result in
\cite[Section~2]{braccinikolovthomas:vkgcdrp22}. In the closest analogue to part~$(b)$ in
\cite[Section~2]{braccinikolovthomas:vkgcdrp22}, $\Omega$ is relatively compact, which is an assumption
we do not make. Thus, a detailed argument is called for.
Let $p,q\in\bdy\Omega$, $p\neq q$. Fix $\kappa\geq 0$. There exists two neighbourhoods $U_p$ of $p$ and $U_q$
of $q$ in $X$ such that $U_p$, $U_q$ are relatively compact in $X$, and such that
$\overline{U_p}\cap\overline{U_q}=\emptyset$. Let $W_p$ and $V_p$ be neighbourhoods of $p$ in $X$ such that
$W_p\Subset V_p\Subset U_p$. 
\medskip

\noindent{{\textbf{Claim.}}}
\emph{There exists a compact set $K\subset \Omega$ such that for every $(1,\kappa)$-almost-geodesic 
$\gamma:[0,T]\lrarw\Omega$, and such that $\gamma(0)\in W_p\cap\Omega$,
$\gamma(T)\in U_q\cap\Omega$, we have $\langle\gamma\rangle\cap K\neq\emptyset$.}
\smallskip

\noindent{\emph{Proof of the Claim.}
If possible let the claim be false. Fix a Hermitian metric
$\hrm$ on $X$ and let $d_{\hrm}$ be the distance induced
by $\hrm$ on $X$. Let $(K_\nu)_{\nu\geq 1}$ be an exhaustion by
compacts of $\Omega$. Then, there exist sequences $(x_\nu)_{\nu\geq 1}\subset W_p\cap\Omega$,
$(y_\nu)_{\nu\geq 1}\subset U_q\cap\Omega$, and $(1,\kappa)$-almost-geodesics 
$\gamma_\nu:[0,T_\nu]\lrarw\Omega$ joining $x_\nu$ and $y_\nu$ for each $\nu\geq 1$, such that
$\langle\gamma_\nu\rangle\cap
K_\nu=\emptyset$. Let $r_\nu>0$ and ${y'_\nu}=\gamma_\nu(r_\nu)$ be such that ${y'_\nu}\in
(U_p\setminus\overline{V_p})\cap\Omega$, and
such that the image of ${\sigma}_{\nu}:=\gamma_\nu{\mid}_{[0,r_\nu]}$ is contained in $U_p$.
By \cite[Lemma~2.9]{masanta:vdekdcm24}, and as we can relabel the subsequences given by it,
there exist points
$\xi\in \overline{W_p}\cap\bdy\Omega$ and $\eta\in\overline{U_p}\cap\bdy\Omega$
such that $x_{\nu}\to\xi$, $y'_{\nu}\to\eta$, and such that
\begin{align}\label{E:distance-boundary}
 \lim_{\nu\to\infty}\sup_{z\in\langle\sigma_\nu\rangle}d_{\hrm}(z,\bdy\Omega)=0.
\end{align}
By construction, $\xi\neq\eta$. Fix $o\in\Omega$. Let $W_{p}^1$, $W_{p}^2$ be neighbourhoods of $p$ in $X$ such that 
$W_p\Subset W_p^1\Subset W_p^2\Subset V_p\Subset U_p.$ Pick $z_\nu\in\langle\sigma_\nu\rangle$ such that
$z_\nu\in W_p^2\setminus\overline{W_p^1}$. Let $s_\nu\in[0,r_\nu]$ be such that $\sigma_\nu(s_\nu)=z_\nu$ for each $\nu\geq 1$.
By \eqref{E:distance-boundary}, there exists a subsequence  $(\nu_k)_{k\geq 1}\subset \N$ and a point
$\tau\in\overline{W_p^2}\cap\bdy\Omega$ such that
$\sigma_{\nu_k}(s_{\nu_k})\to\tau$. By construction, $\tau\neq\xi$ and $\tau\neq\eta$.
Since each $\sigma_\nu$ is a $(1,\kappa)$-almost-geodesic, for each
$\nu\geq 1$ we have,
\begin{align}
  K_\Omega(x_\nu,o)+K_\Omega(o,y'_\nu)
  &\geq K_\Omega(x_\nu,y'_\nu) \notag\\
  &\geq r_\nu-\kappa=(r_\nu-s_\nu)+(s_\nu-0)-\kappa\notag\\
  &\geq K_\Omega(y'_\nu,z_\nu)+K_\Omega(z_\nu,x_\nu)-3\kappa\notag\\
  \implies 2(x_\nu|z_\nu)_o^{\Omega}+2(y'_\nu|z_\nu)_o^{\Omega} &\geq 2K_\Omega(z_\nu,o)-3\kappa.\label{E:gromov-product-inequality}
\end{align}
By hypothesis, there exists a constant $C>0$ such that $\limsup_{k\to\infty} (x_{\nu_k}|z_{\nu_k})_o^{\Omega}+
\limsup_{k\to\infty} (z_{\nu_k}|y'_{\nu_k})_o^{\Omega} < C$. Therefore, by 
\eqref{E:gromov-product-inequality} we have
\begin{align*}
  2\limsup_{k\to\infty}(x_{\nu_k}|z_{\nu_k})_o^{\Omega}+2\limsup_{k\to\infty}(y'_{\nu_k}|z_{\nu_k})_o^{\Omega} &<2C\\
  \implies\limsup_{k\to\infty}2K_\Omega(z_{\nu_k},o)-3\kappa &<2C\\
  \implies\limsup_{k\to\infty}2K_\Omega(z_{\nu_k},o) &<2C+3\kappa.
\end{align*}
Therefore, since $\Omega$ is Kobayashi complete, it follows from the Hopf--Rinow theorem that $\tau\in\Omega$, a contradiction. Hence the
claim.
\hfill $\btl$
} 
\smallskip

By this claim, we have established part~$(b)$ and the result follows.

\end{proof}

\begin{proposition}\label{P:weak-geodesic-visible}
Let $X$ be a complex manifold and let $\Omega\varsubsetneq X$ be a Kobayashi complete domain.
\begin{itemize}[leftmargin=25pt]
  \item[$(a)$] Let $p,q\in\bdy\Omega$, $p\neq q$. If the pair $(p,q)$ satisfies the weak visibility condition, then
  for any $o\in \Omega$,
  \[
    \limsup_{(x,y)\to (p,q)} (x|y)_o^{\Omega} < \infty.
  \]
  \item[$(b)$] Suppose that for some (hence any)
  $o\in \Omega$, $\limsup_{(x,y)\to (p,q)} (x|y)_o^{\Omega} < \infty$ for any $p,q\in\bdy\Omega$, $p\neq q$. Then,
  $\bdy\Omega$ is geodesically visible.
\end{itemize}
\end{proposition}

The argument for part~$(b)$ of Proposition~\ref{P:weak-geodesic-visible} is just a simpler version of the
argument for part~$(b)$ of Proposition~\ref{P:BNT-type_1}. While we do not have a version of 
\cite[Lemma~2.9]{masanta:vdekdcm24} for geodesics to cite, such a version is as obvious as the
latter lemma. As for the proof of part~$(a)$: using the definition of a $(1,\kappa)$-almost-geodesic, one
can adapt the same argument as for \cite[Proposition~2.4]{braccinikolovthomas:vkgcdrp22}, taking care to
incorporate the term $\kappa$ in the relevant estimates. Given these observations,
we shall skip a detailed proof of Proposition~\ref{P:weak-geodesic-visible}.
\medskip

\section{The proofs of Theorems~\ref{T:C-A_A-G} and~\ref{T:equiv}}\label{S:appl_proofs}
We begin with the proof of Theorem~\ref{T:C-A_A-G}.

\begin{proof}[The proof of Theorem~\ref{T:C-A_A-G}]
We shall first prove part~$(a)$. The first step involves showing that $\sigg$\,---\,as given by
Theorem~\ref{T:unit-speed}\,---\,is a $(\lambda,\kappa)$-chord-arc curve. Let $\Gamma$ be as given by
\eqref{E:Gamma}. Our first sub-step is to show that $\Gamma$ is a $(\lambda,\kappa)$-chord-arc curve. Since
$\gamma$ is a $(\lambda,\kappa)$-chord-arc curve, we only need to consider the case when 
$\mathscr{S}(\gamma)$ contains intervals of positive length. Let $\mathcal{I}$ denote the set of intervals in 
$\mathscr{S}(\gamma)$ of positive length. We need to show that 
$\gaux:[0,\tau]\lrarw X$\,---\,which is given by Proposition~\ref{P:aux}\,---\,is a
$(\lambda,\kappa)$-chord-arc curve. By the description of $\gaux$ in Section~\ref{S:key},
$\gamma=\gaux\circ A$, where $A:[0,T]\lrarw[0,\tau]$ is a
continuous, surjective, monotone increasing, piecewise affine
function such that $A\equiv a$ on $[a,b]$ if $[a,b] = \overline{I}$ and
$I \in\mathcal{I}$, and $A$ is strictly increasing on
$[0,T]\setminus\cup_{I\in\mathcal{I}} I$. Fix $s<t$, $s,t\in[0,\tau]$. Let $s',t'\in[0,T]$ be such that
$A(s')=s$ and
$A(t')=t$. Clearly, $\gamma(s')=\gaux(s)$ and $\gamma(t')=\gaux(t)$. Therefore, since $\gamma$ is constant on
$I$ for every $I\in\mathcal{I}$, from the above properties of $A$ it follows that
\begin{align}
  l_X(\gaux{\mid}_{[s,t]})&=l_X(\gamma{\mid}_{[s',t']})\notag\\
  \implies l_X(\gaux{\mid}_{[s,t]})&\leq \lambda K_X(\gaux(s),\gaux(t))+\kappa. \label{E:Gamma-chord-arc}
\end{align}
Since, by Proposition~\ref{P:aux}, $\gaux$ is absolutely continuous, this establishes that $\Gamma$ is a
$(\lambda,\kappa)$-chord-arc curve. By Proposition~\ref{P:unit_speed}-$(b)$, we have
$l_X(\Gamma)<\infty$. Since $\sigg=\Gamma\circ G^{-1}$, where $G$ is as given by \eqref{E:Gamma_length},
\eqref{E:chain_rule} tells us that for every $s<t$, $s,t\in [0,l_X(\Gamma)]$, we have
\begin{align}
  \int_{s}^{t} k_{X}(\sigg(u); \sigg'(u))\,du&= \int_{G^{-1}(s)}^{G^{-1}(t)}k_{X}\big(\Gamma(x);
   \Gamma'(x)\big)\,dx\notag\\
  \implies l_X(\sigg{\mid}_{[s,t]})&=l_X(\Gamma{\mid}_{[G^{-1}(s),G^{-1}(t)]})\notag\\&\leq
  \lambda K_X\big(\Gamma(G^{-1}(s)),\Gamma(G^{-1}(t))\big)+\kappa\label{E:sigg-chord-arc}\\
  &=\lambda K_X(\sigg(s),\sigg (t))+\kappa\notag, 
\end{align}
where the inequality in \eqref{E:sigg-chord-arc} follows from the fact that $\Gamma$ is a 
$(\lambda,\kappa)$-chord-arc curve. Since $\sigg$ is absolutely continuous, this
establishes that $\sigg$ is a $(\lambda,\kappa)$-chord-arc curve. 
\smallskip

Having shown that $\sigg$ is a $(\lambda,\kappa)$-chord-arc curve, given the other properties of
$\sigg$ stated in Theorem~\ref{T:unit-speed}, the remainder of the argument for part~$(a)$ is
exactly the pair of estimates given after \cite[Definition~2.3]{nikolovoktenthomas:lgnvwrKdc24}
(with $\sigg$ in place of $\gamma$).
\smallskip

We shall now prove part~$(b)$. Since $\gamma$ is a $(\lambda,\kappa)$-almost-geodesic, from condition~$(b)$ of
Definition~\ref{D:almost-geodesic} it follows that for all $s<t$, $s,t\in[0,T]$ 
\begin{align*}\label{E:length-inequality}
  l_X(\gamma{\mid}_{[s,t]})&\leq\lambda|s-t|\notag\\
  \implies l_X(\gamma{\mid}_{[s,t]})&\leq\lambda^2(K_X(\gamma(s),\gamma(t))+\kappa),
\end{align*}
where the last inequality follows from condition~$(a)$ of Definition~\ref{D:almost-geodesic}. This
establishes $(b)$.
\end{proof}

In view of the propositions in Section~\ref{S:prop_visi}, we are in a position to give

\begin{proof}[The proof of Theorem~\ref{T:equiv}]
Let $\bdy\Omega$ be weakly visible. Then, from Proposition~\ref{P:weak-geodesic-visible} it follows that
$\bdy\Omega$ is geodesically visible. Conversely, assume $\bdy\Omega$ is geodesically visible. Then, by
Proposition~\ref{P:BNT-type_1}, we have that $\bdy\Omega$ is weakly visible.
\smallskip

Now, let $\bdy\Omega$ be geodesically visible. Let $p,q\in\bdy\Omega$ such that $p\neq q$. There exist two
neighbourhoods $U_p$ of $p$ and $U_q$ of $q$ in $X$ such that
$\overline{U_p}\cap\overline{U_q}=\emptyset$, and a compact set $K\varsubsetneq \Omega$ associated with
$(U_p, U_q)$ by the definition of geodesic visibility.
Let $\gamma:[0,T]\lrarw\Omega$ be an arbitrary hyperbolic geodesic
such that $\gamma(0)\in U_p\cap\Omega$ and $\gamma(T)\in U_q\cap\Omega$. Then, by definition, $\gamma$ is a $(1,0)$-chord-arc curve. 
Therefore, by Theorem~\ref{T:C-A_A-G}-$(a)$, there exists a path $\sigma_\gamma$ in $\Omega$ with
$\langle\gamma\rangle = \langle\sigg\rangle$
such that $\sigg$ is a $(1,0)$-almost-geodesic. Clearly, $\sigg$ is a geodesic. Therefore, since $\bdy\Omega$ is
geodesically visible, 
$\langle\gamma\rangle\cap K\neq \emptyset$. Hence, $\bdy\Omega$ is hyperbolically visible. 
\smallskip

Conversely, assume $\bdy\Omega$ is hyperbolically visible. Let $p,q\in\bdy\Omega$ such that $p\neq q$. 
There exist two relatively compact neighbourhoods $U_p$ of $p$ and $U_q$
of $q$ in $X$ such that
$\overline{U_p}\cap\overline{U_q}=\emptyset$, and a compact set $K\varsubsetneq \Omega$ associated with
$(U_p, U_q)$ by the definition of hyperbolic visibility. 
Fix $o\in\Omega$. Let $x\in U_p\cap\Omega$ and $y\in U_q\cap\Omega$. By Proposition~\ref{P:l_minim}, there exists a
hyperbolic geodesic $\gamma$ in $\Omega$ joining $x$ and $y$. By assumption, $\langle\gamma\rangle\cap K\neq\emptyset$; let
$z\in\langle\gamma\rangle\cap K$. Since $z\in\langle\gamma\rangle$, from Definition~\ref{D:hyp-geo} it follows that
\begin{align*}
  K_\Omega(x,y)=K_\Omega(x,z)+K_\Omega(z,y)&\geq K_\Omega(x,o)+K_\Omega(y,o)-2K_\Omega(z,o)\\
  \implies (x|y)_o^\Omega&\leq \sup_{\zeta\in K}K_\Omega(\zeta,o).
\end{align*}
Hence, $\limsup_{(x,y)\to (p,q)} (x|y)_o^{\Omega}\leq \sup_{\zeta\in K}K_\Omega(\zeta,o) < \infty$. As
$p,q\in \bdy\Omega$ were arbitrarily chosen, by Proposition~\ref{P:weak-geodesic-visible}-$(b)$, 
$\bdy\Omega$ is geodesically visible. Hence, the result follows.
\end{proof}
\smallskip

\section*{Acknowledgements}
R. Masanta is supported by the Theoretical Statistics and Mathematics Unit at the Indian Statistical
Institute, Bangalore Centre. G. Bharali is supported by a DST-FIST grant (grant no.~DST FIST-2021
[TPN-700661]). A part of the work by R. Masanta was done at the Indian Institute of Science; she acknowledges
the support of the DST from a DST-FIST grant (grant no.~DST FIST-2021 [TPN-700661]).
We are very grateful to the anonymous referee for their suggestions on improving the
exposition in several places in this paper.

\end{document}